\documentclass[reqno]{amsart}

\usepackage{times,fancyhdr}
\usepackage[dvips]{graphicx}
\usepackage{amssymb}
\usepackage{amsmath}
\usepackage{amsfonts}
\usepackage{hyperref}

\def \ce{\centering}
\def \R{\mathbb{R}}
\def \P{\mathbb{P}}
\def \N{\mathbb{N}}
\def \E{\mathbb{E}}
\def \D{\mathbb{D}}
\def \T{\mathbb{T}}
\def \H{\mathbb{H}}
\def \L{\mathbb{L}}
\def \C{\mathbb{C}}
\def \Z{\mathbb{Z}}
\def \bf{\textbf}
\def \it{\textit}
\def \sc{\textsc}
\def \bop {\noindent\textbf{Proof }}
\def \eop {\hbox{}\nobreak\hfill
\vrule width 2mm height 2mm depth 0mm
\par \goodbreak \smallskip}
\def \ni {\noindent}
\def \sni {\ss\ni}
\def \bni {\bigskip\ni}
\def \ss {\smallskip}
\def \F{\mathcal{F}}
\def \g{\mathcal{g}}
\def \K{\mathcal{K}}
\vspace{9mm}
\begin{document}

\title[  Controllability  results   ]
{Controllability of stochastic impulsive neutral  functional
differential equations driven by fractional Brownian motion  with
infinite delay}

\author[   Lakhel    ]{  El Hassan  Lakhel  }
\maketitle
\begin{center}{ Cadi Ayyad University,  National School of
Applied Sciences,  46000 Safi, Morocco}\end{center}

 \vspace{0.8cm}
\noindent \textbf{Abstract} In this paper we study the
controllability results of    impulsive neutral stochastic
functional differential equations with infinite delay driven
  by    fractional Brownian motion in a real separable Hilbert space.  The controllability results are obtained using stochastic analysis and  a fixed-point strategy. Finally, an illustrative example
is provided to demonstrate the effectiveness of the theoretical
result.\\




\vspace{.08in} \noindent {\textbf AMS Subject Classification:}
35R10, 93B05  60G22, 60H20.

\def \ce{\centering}
\def \R{\mathbb{R}}
\def \P{\mathbb{P}}
\def \N{\mathbb{N}}
\def \E{\mathbb{E}}
\def \D{\mathbb{D}}
\def \T{\mathbb{T}}
\def \H{\mathbb{H}}
\def \L{\mathbb{L}}
\def \C{\mathbb{C}}
\def \Z{\mathbb{Z}}
\def \bf{\textbf}
\def \it{\textit}
\def \sc{\textsc}
\def \bop {\noindent\textbf{Proof }}
\def \eop {\hbox{}\nobreak\hfill
\vrule width 2mm height 2mm depth 0mm
\par \goodbreak \smallskip}
\def \ni {\noindent}
\def \sni {\ss\ni}
\def \bni {\bigskip\ni}
\def \ss {\smallskip}
\def \F{\mathcal{F}}
\def \g{\mathcal{g}}
\def \K{\mathcal{K}}
 \numberwithin{equation}{section}
\newtheorem{theorem}{Theorem}[section]
\newtheorem{lemma}[theorem]{Lemma}
\newtheorem{proposition}[theorem]{Proposition}
\newtheorem{definition}[theorem]{Definition}
\newtheorem{example}[theorem]{Example}
\newtheorem{remark}[theorem]{Remark}
\allowdisplaybreaks

\section{Introduction}

The notion of controllability is of great importance in mathematical
control theory. Many fundamental problems of control theory such as
pole-assignment, stabilizability and optimal control my be solved
under the assumption that the system is controllable.  The problem
of controllability is to show the existence of control function,
which steers the solution of the system from its initial state to
final state, where the initial and final states may vary over the
entire space.  Conceived by Kalman, the controllability concept has
been studied extensively in the fields of finite-dimensional
systems, infinite-dimensional systems, hybrid systems, and
behavioral systems.  If a system cannot be controlled completely
then different types of controllability can be defined such as
approximate, null, local null and local approximate null
controllability. For more details the reader may refer to
\cite{klam07,klam13,ren11,ren13a,ren13b} and the references therein.
In this paper, we study the controllability  of  neutral functional
stochastic differential equations  of the form\\
{
\begin{equation}\label{eq1}
\begin{split}
d[x(t)-g(t,x_t)]=&[Ax(t)+f(t,x_t)+Bu(t)]dt+\sigma (t)dB^H(t),\, t \in[0, T],\\
& \Delta x|_{t=t_k}=x(t_k^+)-x(t_k^-)=I_k(x(t_k^-)),\; k=1,...,m,\\
x(t)=&\varphi(t)\in L_2^0(\Omega,\mathcal{B}_h),\; for\; a.e.\;
t\in(-\infty,0].
\end{split}
\end{equation}
Here,  $A$ is the infinitesimal generator of an analytic semigroup
of bounded linear operators, $(S(t))_{t\geq 0}$, in a Hilbert space
$X$; $B^H$ is a fractional Brownian motion with $H>\frac{1}{2} $ on
a real and separable Hilbert space $Y$; and the control function
$u(\cdot)$ takes values in $L^2([0,T],U)$, the Hilbert space of
admissible control functions for a separable Hilbert space  $U$; and
$B$ is a bounded linear operator from $U$ into $X$.

\pagestyle{fancy} \fancyhead{} \fancyhead[EC]{E. LAKHEL}
\fancyhead[EL,OR]{\thepage} \fancyhead[OC]{ Controllability for
INSDE Driven by a  Fractional Brownian Motion} \fancyfoot{}
\renewcommand\headrulewidth{0.5pt}

The history $x_{t}:(-\infty,0]\to X$, $x_{t}(\theta)=x(t+\theta)$,
belongs to an abstract phase  space ${\mathcal{B}_h}$ defined
axiomatically,   and $ f,g:[0,T]\times \mathcal{B}_h \to X $ are
appropriated
 functions,    and $\sigma:[0,T] \rightarrow \mathcal{L}_2^0(Y,X)$, are
appropriate functions, where $\mathcal{L}_2^0(Y,X)$ denotes the
space of all $Q$-Hilbert-Schmidt operators from $Y$ into $X$ (see
section 2 below). Moreover, the fixed moments of time $t_k$ satisfy
$0<t_1<t_2<...<t_m< T $; $x(t_k^-)$ and
 $x(t_k^+)$ represent the left and right
 limits of $x(t)$ at time $t_k$ respectively.  $\Delta x(t_k)$
  denotes  the jump in the state $x$
  at time $t_k$ with $I(.): X\longrightarrow X$ determining the size of the
  jump.

The theory of impulsive differential equations as much as neutral
partial differential equations has become an important area of
investigation in recent years stimulated by their numerous
applications to problems arising in  mechanics, medicine and
biology, economics, electronics and telecommunication etc., in which
sudden and abrupt changes occur ingenuously, in the form of
impulses.  For more details on this theory and application see the
papers  \cite{anguraj,lak8,mahkaru14,sakt12,xu2006}.

  It is known that fractional Brownian motion, with Hurst parameter $H\in(0,1)$,
  is a generalization of Brownian motion and it reduces to a standard Brownian motion when
  $H=\frac{1}{2}$. A general theory for the infinite-dimensional stochastic differential equations   driven by
a fractional Brownian motion (fBm) has begun to recieve attention by
various researchers, see e.g. \cite{boufoussi1,lak15,ren13}. For
example, Dung studied the existence and uniqueness of impulsive
stochastic
      Volterra integro-differential equation driven by fBm in
     \cite{dung15} .  Using  the Riemann-Stieltjes integral,
      Boufoussi  et al.  \cite{boufoussi2} proved the existence and
    uniqueness of a mild solution to a related problem and studied the dependence of the
    solution on the initial condition in   infinite dimensional space. Very
    recently, Caraballo and Diop \cite{carab13},  Caraballo et al.   \cite{carab}, and Boufoussi
     and Hajji \cite{boufoussi3} have discussed the existence,
     uniqueness and exponential asymptotic behavior
     of mild solutions by using the Wiener integral.

To the best of the author's knowledge, an investigation concerning
the controllability for impulsive neutral stochastic differential
equations with infinite delay of the form \eqref{eq1} driven by a
fractional Brownian motion has not yet been conducted. Thus, we will
make the first attempt to study such problem in this paper. Our
results are motivated by those in \cite{lak16,lak8} where the
controllability  of mild solutions to  neutral stochastic functional
integro-differential equations driven by fractional Brownian motion
with finite delays   are studied.

The  outline  of this paper is  as follows:  In Section 2 we
introduce some notations, concepts, and basic results about
fractional Brownian motion,  the Wiener integral defined in general
Hilbert spaces, phase spaces  and properties of analytic semigroups
and the fractional powers associated to its generator. In Section 3,
we derive the controllability of impulsive neutral
 stochastic differential systems driven by a fractional Brownian motion.
Finally, in Section 4, we conclude with an example to illustrate the
applicability of the general theory.

\section{Preliminaries}

We collect some notions, concepts and lemmas concerning the Wiener
integral with respect to an infinite dimensional fractional
Brownian,
 and we recall some basic results about analytical semigroups and  fractional powers
 of their infinitesimal generators, which will be used throughout the whole of this chapter.
 For details of the topics addressed in this section, we refer the reader to
\cite{nualart,pazy} and the
  references therein.

Let $(\Omega,\mathcal{F},\{\mathcal{F}_t\}_{t\geq0}, \mathbb{P})$ be
a complete probability space satisfying the usual conditions,
meaning that the filtration is a right-continuous increasing family
and $\mathcal{F}_0$ contains all P-null sets.

 Consider a
time interval $[0,T]$ with arbitrary fixed horizon $T$ and let
$\{\beta^H(t) : t \in [0, T ]\}$ be a one-dimensional fractional
Brownian motion with Hurst parameter $H\in(1/2,1)$. By definition,
$\beta^H$ is a centered Gaussian process with covariance function
$$ R_H(s, t) =\frac{1}{2}(t^{2H} + s^{2H}-|t-s|^{2H}).$$
 Moreover, $\beta^H$ has the following Wiener
integral representation:
\begin{equation}\label{rep}
\beta^H(t) =\int_0^tK_H(t,s)d\beta(s),
 \end{equation}
where $\beta = \{\beta(t) :\; t\in [0,T]\}$ is a Wiener process and
kernel $K_H(t, s)$ is the kernel given by
$$K_H(t, s )=c_Hs^{\frac{1}{2}-H}\int_s^t (u-s)^{H-\frac{3}{2}}u^{H-\frac{1}{2}}du,
$$
for $t>s$, where $c_H=\sqrt{\frac{H(2H-1)}{g (2-2H,H-\frac{1}{2})}}$ and $g(\cdot,\cdot \cdot)$ denotes the Beta function. We take $K_H(t, s ) =0$ if $t\leq s$.\\
We will denote by $\mathcal{H}$ the reproducing kernel Hilbert space
of the fBm. Precisely, $\mathcal{H}$ is the closure of set of
indicator functions $\{1_{[0;t]} : t\in[0,T]\}$ with respect to the
scalar product
$$\langle 1_{[0,t]},1_{[0,s]}\rangle _{\mathcal{H}}=R_H(t , s).$$
The mapping $1_{[0,t]}\rightarrow \beta^H(t)$
 can be extended to an isometry between $\mathcal{H}$
and the first  Wiener chaos and we will denote by $\beta^H(\varphi)$
the image of $\varphi$ by the previous isometry.

Recall that for $\psi,\varphi \in \mathcal{H}$, the scalar product
in $\mathcal{H}$ is given by
$$\langle \psi,\varphi\rangle _{\mathcal{H}}=H(2H-1)\int_0^T\int_0^T\psi(s)\varphi(t)|t-s|^{2H-2}dsdt.
$$
Consider the operator $K_H^*$ from $\mathcal{H}$ to $L^2([0,T])$
defined by
$$(K_H^*\varphi)(s)=\int_s^T\varphi(r)\frac{\partial K_H}{\partial
r}(r,s)dr.
$$ The proof of the fact
that $K_H^*$ is an isometry between $\mathcal{H}$ and $L^2([0,T])$
can be found in \cite{nualart}. Moreover, for any $\varphi \in
\mathcal{H}$, we have
$$\beta^H(\varphi)=\int_0^T(K_H^*\varphi)(t)d\beta(t).
$$

It follows from \cite{nualart} that the elements of $\mathcal{H}$
may be not functions but rather distributions of negative order. In
order to obtain a space of functions contained in $\mathcal{H}$, we
consider the linear space $|\mathcal{H}|$ generated by the
measurable functions $\psi$ such that
$$\|\psi \|^2_{|\mathcal{H}|}:= \alpha_H  \int_0^T \int_0^T|\psi(s)||\psi(t)| |s-t|^{2H-2}dsdt<\infty,
$$
where $\alpha_H = H(2H-1)$. We have the following lemma  (see
\cite{nualart}).
\begin{lemma}\label{lem1}
The space $|\mathcal{H}|$ is a Banach space with the norm
$\|\psi\|_{|\mathcal{H}|}$; the following inclusions hold
$$\mathbb{L}^2([0,T])\subseteq \mathbb{L}^{1/H}([0,T])\subseteq |\mathcal{H}|\subseteq \mathcal{H};
$$
and for any $\varphi\in \mathbb{L}^2([0,T])$,
$$\|\psi\|^2_{|\mathcal{H}|}\leq 2HT^{2H-1}\int_0^T
|\psi(s)|^2ds.
$$
\end{lemma}
Let $X$ and $Y$ be two real, separable Hilbert spaces and let
$\mathcal{L}(Y,X)$ be the space of bounded linear operator from $Y$
to $X$. For convenience, we shall use the same notation to denote
the norms in $X,Y$ and $\mathcal{L}(Y,X)$. Let $Q\in
\mathcal{L}(Y,Y)$ be an operator defined by $Qe_n=\lambda_n e_n$
with finite trace
 $trQ=\sum_{n=1}^{\infty}\lambda_n<\infty$, where $\lambda_n \geq 0 \; (n=1,2...)$ are nonnegative
  real numbers and $\{e_n : n=1,2...\}$ is a complete orthonormal basis in $Y$.
 Let $B^H=(B^H(t))$  be  a $Y-$ valued fbm on
  $(\Omega,\mathcal{F}, \mathbb{P})$ with covariance $Q$ given by
 $$B^H(t)=B^H_Q(t)=\sum_{n=1}^{\infty}\sqrt{\lambda_n}e_n\beta_n^H(t),
 $$
 where $\beta_n^H$ are real, independent fBm's. This process is  Gaussian, it
 starts from $0$, has zero mean, and covariance:
 $$E\langle B^H(t),x\rangle\langle B^H(s),y\rangle=R(s,t)\langle Q(x),y\rangle \;\; \mbox{for all}\; x,y \in Y \;\mbox {and}\;  t,s \in [0,T].
 $$
In order to define Wiener integrals with respect to the $Q$-fBm, we
introduce the space $\mathcal{L}_2^0:=\mathcal{L}_2^0(Y,X)$  of all
$Q$-Hilbert-Schmidt operators $\psi:Y\rightarrow X$. Recall that
$\psi \in \mathcal{L}(Y,X)$ is called a $Q$-Hilbert-Schmidt operator
if
$$  \|\psi\|_{\mathcal{L}_2^0}^2:=\sum_{n=1}^{\infty}\|\sqrt{\lambda_n}\psi e_n\|^2 <\infty,
$$
and that the space $\mathcal{L}_2^0$ equipped with the inner product
$\langle \varphi,\psi
\rangle_{\mathcal{L}_2^0}=\sum_{n=1}^{\infty}\langle \varphi
e_n,\psi e_n\rangle$ is a separable Hilbert space.

Let $\phi: [0,T] \rightarrow \mathcal{L}_2^0(Y,X)$ be a given
function.
 The Wiener integral of $\phi$ with respect to $B^H$ is defined by

{\footnotesize{\begin{equation}\label{int}
\int_0^t\phi(s)dB^H(s)=\sum_{n=1}^{\infty}\int_0^t
\sqrt{\lambda_n}\phi(s)e_nd\beta^H_n(s)=\sum_{n=1}^{\infty}\int_0^t
 \sqrt{\lambda_n}(K_H^*(\phi e_n)(s)d\beta_n(s),
\end{equation}}}
where $\beta_n$ is the standard Brownian motion used to  define $\beta_n^H$ as in $(\ref{rep})$.\\
We conclude this subsection by stating the following result which is
critical in the proof of our result. It can be proved using
arguments similar to those used to prove   Lemma 2 in \cite{carab}.
\begin{lemma}\label{lem2}
If $\psi:[0,T]\rightarrow \mathcal{L}_2^0(Y,X)$ satisfies $\int_0^T
\|\psi(s)\|^2_{\mathcal{L}_2^0}ds<\infty,$
 then $(\ref{int})$ is well-defined as an $X$-valued random variable and
$$ \mathbb{E}\|\int_0^t\psi(s)dB^H(s)\|^2\leq 2Ht^{2H-1}\int_0^t \|\psi(s)\|_{\mathcal{L}_2^0}^2ds.
 $$
\end{lemma}

It is known that the study of theory of  differential equation with
infinite delays depends on a choice of the abstract phase space. We
assume that the phase space $\mathcal{B}_h$ is a linear space of
functions mapping $(-\infty ,0]$ into $X$, endowed with a norm
$\|.\|_{\mathcal{B}_h}$.  We shall introduce some basic definitions,
notations and lemma which are used in this paper. First, we present
the abstract phase space $\mathcal{B}_h$. Assume that $h:(-\infty
,0]\longrightarrow [0,+\infty)$  is a continuous function with
$l=\int_{-\infty}^0h(s)ds<+\infty$.

We define the abstract phase space $\mathcal{B}_h$ by
$$
\begin{array}{ll}
                    \mathcal{B}_h=   & \{\psi:(-\infty,0]\longrightarrow X \text{ for any } \tau>0, (\E\|\psi\|^2)^{\frac{1}{2}} \text{ is bounded and measurable }  \\
                       & \text{ function on } [-\tau,0] \text{ and } \int_{-\infty}^0h(t)\sup_{t\leq s\leq0}
                       (\E\|\psi(s)\|^2)^{\frac{1}{2}}dt<+\infty\}.
                    \end{array}
$$
If we equip this space with the norm

$$
\|\psi\|_{\mathcal{B}_h}:=\int_{-\infty}^0h(t)\sup_{t\leq
s\leq0}(\E\|\psi\|^2)^{\frac{1}{2}}dt,
$$
then it is clear that $(\mathcal{B}_h,\|.\|_{\mathcal{B}_h})$  is a
Banach space.

We now consider the space $\mathcal{B}_{DI}$ ($D$ and $I$ stand for
delay and impulse, respectively ) given by

$$\begin{array}{ll}
    \mathcal{B}_{DI} &= \{ x:(-\infty,T]\to X: x|I_k\in \mathcal{C}(I_k,X)
\text{ and } x(t_k^+), x(t_k^-) \text{ exist with  }\\ \\
     & x(t_k^-)=x(t_k), k=1,2,...,m,   x_0=\varphi\in \mathcal{B}_h  \text{
     and }
\sup_{0\leq t \leq T}\E(\|x(t)\|^2)<\infty \},
  \end{array}
$$
where $x|I_k$ is the restriction of $x$ to the interval
$I_k=(t_k,t_{k+1}]$, $k=1,2,...,m.$ The function
$\|.\|_{\mathcal{B}_{DI}}$ to be a semi-norm in $\mathcal{B}_{DI}$,
it is defined by
$$
\|x\|_{\mathcal{B}_{DI}}=\|x_0\|_{\mathcal{B}_h}+\sup_{0\leq t \leq
T}(\E(\|x(t)\|^2))^{\frac{1}{2}}.
$$
The following lemma  is a common property of phase spaces.
\begin{lemma}\cite{li-li} Suppose $x\in \mathcal{B}_{DI}$, then for all $t\in [0,T]$ ,  $x_t\in
\mathcal{B}_h$  and
$$
l(\E\|x(t)\|^2)^{\frac12}\leq\|x_{t}\|_{\mathcal{B}_h}\leq l\sup_{0
\leq s\leq t}(\E\|x(s)\|^2)^{\frac12}+ \|x_{0 }\|_{\mathcal{B}_h},
$$
where $l=\int_{-\infty}^0h(s)ds<\infty$.

\end{lemma}

\begin{definition} The map $f:[0,T]\times \mathcal{B}_h\to X$ is said
to be an $L^2$-Carath\'eodory if
\begin{enumerate}
  \item  $t\longmapsto f(t,x)$ is measurable for each $x\in
  \mathcal{B}_h$,
  \item $x\longmapsto f(t,x)$ is continuous for almost all $t\in
  [0,T]$,
  \item for every positive integer $q$ there exists $h_q\in
  L^1([0,T],\R^+)$, such that
  $$
\E\|f(t,x)\|^2\leq h_q(t), \,\text{ for all  }\,
\|x\|_{\mathcal{B}_h}^2\leq q \text{ and for a.e.  } \,t\in[0,T].
  $$
\end{enumerate}

\end{definition}

 Next, we introduce some notations and basic facts about
the theory of semigroups and fractional power operators. Let
$A:D(A)\rightarrow X$  be the infinitesimal generator of an analytic
semigroup, $(S (t))_{t\geq 0}$, of bounded linear operators on $X$.
The theory of strongly continuous is thoroughly discussed in
\cite{pazy}  and \cite{gold}.  It is well-known that there exist $M
\geq 1$ and $\lambda \in \mathbb{R}$ such that $\|S (t)\|\leq M
e^{\lambda t}$ for every $t \geq 0$.
  If $(S (t))_{t\geq 0}$ is a
uniformly bounded, analytic semigroup such that $0 \in \rho(A)$,
where $\rho (A)$ is the resolvent set of $A$, then it is possible to
define the fractional power $(-A)^{\alpha}$  for $0 < \alpha \leq
1$, as a closed linear operator on its domain $D(-A)^{\alpha}$.
Furthermore, the subspace $D(-A)^{\alpha}$ is dense in $X$, and the
expression
$$\|h\|_{\alpha} =\|(-A)^{\alpha}h\|$$ defines a norm in
$D(-A)^{\alpha}$. If $X_{\alpha}$ represents the space
$D(-A)^{\alpha}$ endowed with the norm $\|.\|_{\alpha}$, then the
following properties hold (see \cite{pazy}, p. 74).
\begin{lemma}\label{lem3} Suppose that $A, X_{\alpha},$ and $(-A)^{\alpha}$ are as described above.
\begin{itemize}
\item[(i)] For \,$0<\alpha \leq  1$,  $X_{\alpha}$ is a Banach space.
\item[(ii)] If\,  $0 <\beta \leq \alpha, $ then the injection $X_{\alpha}
\hookrightarrow X_{\beta}$ is continuous.
 \item[(iii)] For every \,$0<\alpha \leq 1,$ there exists $M_{\alpha} > 0 $ such that
 $$\| (-A)^{\alpha}S (t)\|\leq M_{\alpha}t^{-\alpha}e^{-\lambda t}
 , \;\;\;\;t>0,\;\; \lambda>0 .$$
 \end{itemize}
 \end{lemma}

\section{Controllability Result}

In this section we derive controllability conditions for a class of
neutral stochastic functional differential equations with infinite
delays driven   by    a fractional Brownian motion in a real
separable Hilbert space. Before starting, we introduce the concepts
of a mild solution of the problem (\ref{eq1})  and the meaning of
controllability of
 neutral stochastic functional differential equation.

\begin{definition}
An $X$-valued  process $\{x(t) : t\in(-\infty,T]\}$ is a mild
solution of (\ref{eq1}) if

\begin{enumerate}
\item  $x(t)$ is measurable for each $t>0$, $x(t)=\varphi(t)$ on $( -\infty, 0]$, $\Delta x|_{t=t_k}=I_k(x(t_k^-))$, $k=1,2,...,m,$
 the restriction of $x(.)$ to $[0,T]-\{t_1,t_2,...,t_m\}$ is continuous,
 and for each $s\in[0,t)$ the function  $AS(t-s)g(s,x_s)$ is
 integrable,
\item for arbitrary $t \in [0,T]$, we have

\begin{equation}\label{eqmild2}
\begin{array}{ll}
x(t)&=S(t)(\varphi(0)-g(0,\varphi ))+g(t,x_t)\\ \\
&+ \int_0^t AS(t-s)g(s,x_s)ds +\int_0^t S(t-s)f(s,x_s)ds\\ \\
&+\int_0^t S(t-s)Bu(s)ds+\int_0^t S(t-s)\sigma(s)dB^H(s),\;\;\;\\
\\
 &+{\sum}_{0<t_k<t}S(t-t_k)I_k(x(t_k^-)),\;\;\; \mathbb{P}-a.s.
\end{array}
\end{equation}
\end{enumerate}
\end{definition}
\begin{definition}
The impulsive  neutral  stochastic functional differential equation
(\ref{eq1}) is said to be controllable on the interval $(-\infty,T]$
if for every initial stochastic process $\varphi$ defined on
$(-\infty,0]$, there exists a stochastic control $u\in L^2([0,T],
U)$ such that the mild solution $x(\cdot)$ of  (\ref{eq1}) satisfies
$x(T)=x_1$, where $x_1$ and $T$ are the preassigned terminal state
and time, respectively.
\end{definition}

Our main result in this paper is based on the following fixed point
theorem.
\begin{theorem}(Karasnoselskii's fixed point theorem)
Let $V$ be a bounded closed and convex subset of  a Banach space $X$
 and  let $\Pi_1$, $\Pi_2$ be two operators of $V$
 into $X$ satisfying:
\begin{enumerate}
\item  $\Pi_1(x)+\Pi_2(x)\in V$ whenever $x\in V$,

  \item  $\Pi_1$ is a contraction mapping, and
  \item $\Pi_2$ is  completely continuous.
\end{enumerate}
Then,  there exists a $z \in V$ such that $z=\Pi_1(z)+\Pi_2(z)$.
\end{theorem}

 In order to establish the controllability of
(\ref{eq1}), we impose the following conditions on the data of the
problem:

\begin{itemize}
\item [$(\mathcal{H}.1)$]
 $A$ is the infinitesimal generator of an analytic semigroup, $(S (t))_{t\geq
0}$, of bounded linear operators on $X$. Further, $0 \in \rho (A)$,
and there exist constants $M,\; M_{1-\beta}$ such that
$$ \| S (t)\|^2\leq M\; \;\;\mbox{and}\;\;
\|(-A)^{1-\beta}S(t)\|\leq \frac{M_{1-\beta}}{t^{1-\beta}}, \,
\text{for all } t \in [0,T]$$ (see Lemma \ref{lem3}).

\item [$(\mathcal{H}.2)$]  $f$ is $L^2$-Carath\'eodory map and     there exist positive constants
$M_f, \overline{M_f}$  for $t\in[0,T]$, $x_1,x_2\in\mathcal{B}_{h}$
such that
$$
\E\|f(t,x_1)-f(t,x_2)\|^2\leq M_f\|x_1-x_2\|^2_{\mathcal{B}_{h}}, \,
\text{ and }\qquad \overline{M_f}=\sup_{t\in[0,T]}\|f(t,0)\|^2.
$$

\item [$(\mathcal{H}.3)$] There exist constants $0 <
\beta < 1,$\; $M_g>0$  and $\nu>0$  such
 that the function $g$ is $X_{\beta}$-valued and satisfies
\begin{itemize}
 \item [i)] $\E\|(-A)^{\beta}g(t,x)-(-A)^{\beta}g(t,y)\|^2 \leq M_g \|x-y\|^2_{\mathcal{B}_{h}},\; t\in [0,T], \;  $

  $x,y\in\mathcal{B}_{h}$   with \; $ \nu=4M_g
l^2(\|(-A)^{-\beta}\|^2+\frac{(M_{1-\beta}T^\beta)^2}{2\beta-1})<1.
$
\item [ii)]  $c_1= \|(-A)^{-\beta}\|$\; and
$\overline{M}_g=\sup_{t\in[0,T]}\|(-A)^{-\beta}g(t,0)\|^2$.
\end{itemize}

\item [$(\mathcal{H}.4)$] $I_k: X\longrightarrow X$   $k=1,2,...,m,$ and  there exist  constants
$M_k\geq0,$ $\widetilde{M}_k\geq0$   such that
  $\E\|I_k(x)-I_k(y)\|^2\leq M_k\|x-y\|^2$  and $\|I_k(x)\|^2\leq \widetilde{M}_k$  for
any $x, y\in X$.

\item [$(\mathcal{H}.5)$] The function $\sigma:[0,\infty)\rightarrow \mathcal{L}_2^0(Y,X)$ satisfies
 $$\int_0^T\|\sigma(s)\|^2_{\mathcal{L}_2^0}ds< \infty,\;\; \forall T>0 .
 $$
 \item [$(\mathcal{H}.6)$] The linear operator $W$ from $U$ into $X$
 defined by
 $$
Wu=\int_0^TS(T-s)Bu(s)ds
 $$
 has an inverse operator $W^{-1}$ that takes values in $L^2([0,T],U)\setminus ker
 W$, where $$
 ker
 W=\{x\in L^2([0,T],U):  \; W x=0\}
 $$ (see \cite{klam07}), and there
 exists finite
 positive constants $M_b,$ $M_w$ such that $\|B\|^2\leq M_b$ and $\|W^{-1}\|^2\leq M_w.$

\end{itemize}

The main result of this chapter is the following.
\begin{theorem}\label{th1}
Suppose that $(\mathcal{H}.1)-(\mathcal{H}.6)$ hold. Then,  the
system  (\ref{eq1}) is controllable  on $(-\infty,T]$ provide that
\begin{equation}\label{cond1}
7l^2(1+8MM_bM_wT^2)\{8(c_1^2+\frac{(M_{1-\beta}T^\beta)^2}{2\beta-1})M_g+8MT^2M_f\}<1.
\end{equation}

\end{theorem}
\begin{proof} Transform the problem(\ref{eq1}) into a fixed-point
problem. To do this, using the hypothesis $(\mathcal{H}.6)$ for an
arbitrary function $x(\cdot)$, define the control by

$$
\begin{array}{lll}
  u(t) & =&W^{-1}\{x_1-S(T)(\varphi(0)-g(0,x_0))-g(T,x_T))\\\\
   & -&\int_0^T AS(T-s)g(s,x_s)ds -\int_0^T S(T-s)f(s,x_s)ds\\\\
   & -& \int_0^T
   S(T-s)\sigma(s)dB^H(s)\}(t)-\sum_{0<t_k<T}S(T-t_k)I_k(x(t_k^-))\}(t).\\
\end{array}
$$

To formulate the controllability problem in the form suitable for
application of the Banach fixed point theorem, put the control
$u(.)$ into the stochastic control system (\ref{eqmild2}) and obtain
a non linear  operator $\Pi$ on $\mathcal{B}_{DI}$ given  by
$$ \Pi(x)(t)=\left\{
\begin{array}{ll}
& \varphi(t), \;\;\; \text {if } \; t\in(-\infty,0], \\\\
&S(t)(\varphi(0)-g(0,\varphi ))+g(t,x_t)\\ \\
&+ \int_0^t AS(t-s)g(s,x_s)ds +\int_0^t S(t-s)f(s,x_s)ds\\ \\
&+\int_0^t S(t-s)Bu(s)ds+\int_0^t S(t-s)\sigma(s)dB^H(s),\;\;\;\\
\\
 &+{\sum}_{0<t_k<t}S(t-t_k)I_k(x(t_k^-)),\;\;\; \text {if }\;
 t\in[0,T].
\end{array}\right.
$$

 Then it is clear that to prove the existence of mild
solutions to equation (\ref{eq1}) is equivalent to find a fixed
point for the operator $\Pi$. Clearly, $\Pi x(T)=x_1$, which means
that the control $u$ steers the system from the initial state
$\varphi$ to $x_1$ in time $T$, provided we can obtain a fixed point
of the operator $\Pi$ which implies that the system in controllable.

 Let $y:(-\infty,T]\longrightarrow X$ be the function defined by
 $$
y(t)=\left\{\begin{array}{ll}
       \varphi(t), & \text {if } \; t\in(-\infty,0], \\
       S(t)\varphi(0), & \text {if }\; t\in[0,T],
     \end{array}\right.
 $$
then, $y_0=\varphi$. For each function $z\in \mathcal{B}_{DI}$, set
$$
x(t)=z(t)+y(t).
$$

It is obvious that $x$ satisfies the stochastic control system
(\ref{eqmild2}) if and only if $z$ satisfies $z_0=0$ and

\begin{equation}\label{eqmild2}
\begin{array}{ll}
z(t)=&g(t,z_t+y_t)-S(t)g(0,\varphi ))+ \int_0^t
AS(t-s)g(s,z_s+y_s)ds
\\\\
&+\int_0^t S(t-s)f(s,z_s+y_s)ds+\int_0^t S(t-s)Bu_{z+y}(s)ds\\\\
&+\int_0^t S(t-s)\sigma(s)dB^H(s),\;\;\;\\
\\
 &+{\sum}_{0<t_k<t}S(t-t_k)I_k(z(t_k^-)+y(t_k^-)).
\end{array}
\end{equation}

Set
$$
\mathcal{B}_{DI}^0=\{z\in \mathcal{B}_{DI}: z_0=0\};
$$
for any $z\in B_{DI}^0$, we have
$$
\|z\|_{\mathcal{B}_{DI}^0}=\|z_0\|_{\mathcal{B}_h}+
\sup_{t\in[0,T]}(\E\|z(t)\|^2)^{\frac{1}{2}}=\sup_{t\in[0,T]}(\E\|z(t)\|^2)^{\frac{1}{2}}.
$$
Then, $(\mathcal{B}_{DI}^0,\|.\|_{\mathcal{B}_{DI}^0})$ is a Banach
space. Define the operator $\widehat{\Pi}:
\mathcal{B}_{DI}^0\longrightarrow \mathcal{B}_{DI}^0$ by

\begin{equation}
(\widehat{\Pi}z)(t)=\left\{\begin{array}{ll}
&0\;\;\text {if } \; t\in(-\infty,0], \\\\
&g(t,z_t+y_t)-S(t)g(0,\varphi ))+ \int_0^t AS(t-s)g(s,z_s+y_s)ds \\\\
&+\int_0^t S(t-s)f(s,z_s+y_s)ds+\int_0^t S(t-s)Bu_{z+y}(s)ds\\\\
&+\int_0^t S(t-s)\sigma(s)dB^H(s),\;\;\;\\
\\
 &+{\sum}_{0<t_k<t}S(t-t_k)I_k(z(t_k^-)+y(t_k^-)),\;\;\;\;\text {if } \;
 t\in[0,T],
\end{array}\right.
\end{equation}

where
 $$
\begin{array}{ll}
  u_{z+y}(t)
  &=W^{-1}\{x_1-S(T)(\varphi(0)-g(0,z_0+y_0))-g(T,z_T+y_T))\\\\
    &-\int_0^T AS(T-s)g(s,z_s+y_s)ds -\int_0^T
    S(T-s)f(s,z_s+y_s)ds\\\\
    & -\int_0^T
   S(T-s)\sigma(s)dB^H(s)\}(t)-\sum_{0<t_k<T}S(T-t_k)I_k(z(t_k^-)+y(t_k^-))\}(t).\\
\end{array}
$$

Set
$$
\mathcal{B}_k=\{z\in \mathcal{B}_{DI}^0:
\|z\|_{\mathcal{B}_{DI}^0}^2\leq k\}, \qquad \text{ for some }
k\geq0,
$$
then $\mathcal{B}_k\subseteq \mathcal{B}_{DI}^0$ is a bounded closed
convex set, and for $z\in \mathcal{B}_k$, we have
$$
\begin{array}{ll}
  \|z_t+y_t\|_{\mathcal{B}_{DI}} & \leq 2(\|z_t\|^2_{\mathcal{B}_{DI}}+\|y_t\|^2_{\mathcal{B}_{DI}})
  \\\\
   &\leq 4(l^2\sup_{0\leq s\leq t}\E\|z(s)\|^2+\|z_0\|^2_{\mathcal{B}_{h}}\\\\
   &+l^2\sup_{0\leq s\leq
   t}\E\|y(s)\|^2+\|y_0\|^2_{\mathcal{B}_{h}})\\\\
   &\leq 4l^2(k+M\E\|\varphi(0)\|^2) +4\|y\|_{\mathcal{B}_h}\\\\
   &:=q'.
\end{array}
$$
From our assumptions, using the fact that $(\sum_{i=1}^n a_i)^2 \leq
n \sum_{i=1}^n a_i^2 $ for any positive real numbers $a_i$,
$i=1,2,...,n,$, we have
\begin{equation}\label{estim0}
\begin{array}{ll}
  \E\|u_{z+y}\|^2 \leq&
  8M_w\{\|x_1\|^2+M\E\|\varphi(0)\|^2+2Mc_1^2M_g\|y\|^2_{\mathcal{B}_h}\\\\
   & +2(c_1^2+\frac{(M_{1-\beta}T^\beta)^2}{2\beta-1})[M_gq'+\overline{M}_g]+2MT^2[M_fq'+\overline{M_f}]\\ \\
   &+2MT^{2H-1}\int_0^T\|\sigma(s)\|^2_{\mathcal{L}_2^0}ds+mM\sum_{k=1}^m\widetilde{M}_k\}:=\mathcal{G}.
\end{array}
\end{equation}
Noting that
\begin{equation}\label{estim1}
\begin{array}{ll}
  \E\|u_{z+y}-u_{v+y}\|^2 & \leq 4M_w\{(c_1^2+\frac{(M_{1-\beta}T^\beta)^2}{2\beta-1})M_g+MT^2M_f\\ \\
   &+mM\sum_{k=1}^m M_k\} \|z_t-v_t\|_{\mathcal{B}_h}^2.\\
\end{array}
\end{equation}
 It is clear that the operator $\Pi$ has a fixed point if
and only if $\widehat{\Pi}$ has one, so it turns to prove that
$\widehat{\Pi}$ has a fixed point. To this end, we decompose
$\widehat{\Pi}$ as $\widehat{\Pi}=\Pi_1+\Pi_2$, where $\Pi_1$ and
$\Pi_2$ are defined on $\mathcal{B}_{DI}^0$, respectively by
\begin{equation} (\Pi_1z)(t)=\left\{\begin{array}{ll}
&0\;\;\text {if } \; t\in(-\infty,0], \\\\
&g(t,z_t+y_t)-S(t)g(0,\varphi ))+ \int_0^t
AS(t-s)g(s,z_s+y_s)ds\\\\
&+\int_0^t S(t-s)\sigma(s)dB^H(s),\;\;\;\;\text {if } \; t\in[0,T],
\end{array}\right.
\end{equation}
\\
and \\
 \begin{equation}
(\Pi_2z)(t)=\left\{\begin{array}{ll}
&0\;\;\text {if } \; t\in(-\infty,0], \\\\
& \int_0^t S(t-s)f(s,z_s+y_s)ds+\int_0^t
S(t-s)Bu_{z+y}(s)ds\\\\
&+{\sum}_{0<t_k<t}S(t-t_k)I_k(z(t_k^-)+y(t_k^-)),\;\;\;\;\text {if }
\; t\in[0,T].
\end{array}\right.
\end{equation}

In order to apply the Karasnoselskii fixed point theorem for the
operator $\widehat{\Pi}$, we prove the following assertions:
\begin{enumerate}
\item $\Pi_1(x)+\Pi_2(x)\in \mathcal{B}_k$ whenever $x\in \mathcal{B}_k$,
  \item   $\Pi_1$ is a contraction;
  \item $\Pi_2$ is continuous and  compact map.
  \end{enumerate}

For the sake of convenience,  the proof will be given in several
steps.

Step 1: We claim that there exists a positive number $k$, such that
$\Pi_1(x)+\Pi_2(x)\in \mathcal{B}_k$ whenever $x\in \mathcal{B}_k$.
If it is not true, then for each positive number $k$, there is a
function $z^k(.)\in\mathcal{B}_k$, but  $\Pi_1(z^k)+\Pi_2(z^k)\notin
\mathcal{B}_k$, that is $\E\|\Pi_1(z^k)(t)+\Pi_2(z^k)(t)\|^2>k$ for
some $t\in [0,T].$  However, on the other hand, we have
$$
\begin{array}{ll}
  k & <\E \|\Pi_1(z^k)\Pi_2(z^k)(t)\|^2 \\\\
   & \leq 7\{2Mc_1^2(M_g\|y\|_{\mathcal{B}_h}^2+\overline{M}_g)+2(c_1^2q'+\overline{M}_g)
   +2\frac{(M_{1-\beta}T^\beta)^2}{2\beta-1}[M_gq'+\overline{M}_g]\\\\
   &+MM_bT^2\mathcal{G} +2MT^2(M_fq'+\overline{M_f})+2MT^{2H-1}\int_0^T\|\sigma(s)\|^2_{\mathcal{L}_2^0}ds+ M\sum_{k=1}^m \widetilde{M}_k\}\\\\
   &\leq
   7(1+8MM_bM_wT^2)\{2Mc_1^2(M_g\|y\|_{\mathcal{B}_h}^2+\overline{M}_g)+
   2(c_1^2+\frac{(M_{1-\beta}T^\beta)^2}{2\beta-1})[M_gq'+\overline{M}_g])\\\\
   &+2MT^2[M_fq'+\overline{M_f}]+2MT^{2H-1}\int_0^T\|\sigma(s)\|^2_{\mathcal{L}_2^0}ds+mM\sum_{k=1}^m\widetilde{M}_k\}\\\\
   &+ 8MM_bM_wT^2(\|x_1\|^2+M\E\|\varphi(0)\|^2)\\\\
   &\leq \overline{K}+
   7(1+8MM_bM_wT^2)\{c_1^2+2\frac{(M_{1-\beta}T^\beta)^2}{2\beta-1})M_gq'+2MT^2M_fq'\},
\end{array}
$$
 where $\overline{K}$ is independent of $k$. Dividing both sides by
 $k$ and taking the limit as $k\longrightarrow\infty$, we get
 $$
7l^2(1+8MM_bM_wT^2)\{8(c_1^2+\frac{(M_{1-\beta}T^\beta)^2}{2\beta-1})M_g+8MT^2M_f\}\geq1.
 $$

 This contradicts  (\ref{cond1}). Hence for some positive $k$,
 $$
(\Pi_1+\Pi_2) (\mathcal{B}_k)\subseteq \mathcal{B}_k.
 $$

Step 2: $\Pi_1$ is a contraction.\\
Let $t\in[0,T]$ and $z^1, z^2\in\mathcal{B}_{DI}^0$
$$
 \begin{array}{ll}
             \E\|(\Pi_1z^1)(t)-(\Pi_1z^2)(t) \|^2& \leq 2\E\|g(t,z_t^1+y_t)-g(t,z_t^2+y_t)\|^2\\ \\
              & +2\E\|\int_0^tAS(t-s)(g(s,z^1_s+y_s)-g(s,z^2_s+y_s))\|^2
              \\\\
              & \leq 2M_g\|(-A)^{-\beta}\|^2\|z_s^1-z_s^2\|^2_{\mathcal{B}}\\ \\
              & +2T\int_0^t\frac{M^2_{1-\beta}}{(t-s)^{2(1-\beta)}}M_g
              \|z_s^1-z_s^2\|^2_{\mathcal{B}}\\\\
              & \leq  2M_g\left\{\|(-A)^{-\beta}\|^2+\frac{(M_{1-\beta}T^\beta)^2}{(2\beta-1)}\right\}(2l^2\sup_{0\leq s\leq T}\\\\
              &\E\|z^1(s)-z^2(s)\|^2 +2(\|z^1_0\|^2_{\mathcal{B}}+\|z^2_0\|^2_{\mathcal{B}})
              \\\\
              & \leq \nu\sup_{0\leq s\leq T}\E\|z^1(s)-z^2(s)\|^2)
              \qquad ( \text{ since }\;z^1_0=z^2_0=0)
           \end{array}
$$
Taking supremum over $t$,
$$
\|(\Pi_1z^1)(t)-(\Pi_1z^2)(t) \|_{\mathcal{B}_{DI}^0}\leq
\nu\|z^1-z^2\|_{\mathcal{B}_{DI}^0},
$$
where
$$
\nu=4M_gl^2(\|(-A)^{-\beta}\|^2+\frac{(M_{1-\beta}T^\beta)^2}{2\beta-1})<1.
$$
Thus $\Pi_1$ is a contraction on $\mathcal{B}_{DI}^0$.

Step 3: $\Pi_2$ is completely continuous $\mathcal{B}_{DI}^0$.
\begin{enumerate}
  \item $\Pi_2$ is  continuous on $\mathcal{B}_{DI}^0$.\\

  Let $z^n$ be a sequence such that $z^n\longrightarrow z$ in
  $\mathcal{B}_{DI}^0$. Then, there exists a number $k>0$ such that $\|z^n(t)\|\leq k$, for all $n$ and a.e.
    $t\in[0,T]$, so $z^n\in \mathcal{B}_k$ and $z\in \mathcal{B}_k$.
    In view  of \ref{estim1}, we have

\begin{equation*}
\begin{array}{ll}
  \E\|u_{z+y}-u_{v+y}\|^2 & \leq 8k M_w\{(c_1^2+\frac{(M_{1-\beta}T^\beta)^2}{2\beta-1})M_g+MT^2M_f\\ \\
   &+mM\sum_{k=1}^m M_k\} .\\
\end{array}
\end{equation*}
   By  hypothesis $H1-H4$, we have
   \begin{itemize}
     \item [(i)] $I_k$, $k=1,2,...,m$ is continuous.\\
     \item [(ii)] $f(t,z_t^n+y_t)\longrightarrow f(t,z_t+y_t)$ for
     each $t\in[0,T]$.\\
     \item [(iii)] $g(t,z_t^n+y_t)\longrightarrow g(t,z_t+y_t)$ for
     each $t\in[0,T]$.\\
     \item [(iv)] $AS(T-s)g(s,z_s^n+y_s)\longrightarrow  AS(T-s)g(s,z_s+y_s)$ for
     each $s\in[0,T]$.
   \end{itemize}
We have by the dominated convergence theorem,
$$
\begin{array}{ll}
  \E\|(\Pi_2z)^n(t)-(\Pi_2z)(t)\|^2 & \leq 3\{\E\|\int_0^tS(t-s)B[z^n(s)-z(s)]ds\|^2\\ \\
   & +\E\|\int_0^tS(t-s)[f(s,z_s^n+y_s)-f(s,z_s+y_s)]ds\|^2 \\\\
   & +\E\|\sum_{0\leq t_k\leq t}
   S(t-t_k)[I_k(z^n({t_k^-})+y({t_k^-}))-I_k(z(t_k^-)+y(t_k^-))]\|^2\}
   \\\\
  & \leq 3MM_bT\int_0^t\E\|z^n(s)-z(s)\|^2 ds\\\\
  &+3MT\int_0^t\E\|f(s,z_s^n+y_s)-f(s,z_s+y_s)\|^2ds\\ \\ \\
   & +3mM\sum_{k=1}^m\E\|I_k(z^n({t_k^-})+y({t_k^-}))-I_k(z({t_k^-})+y({t_k^-}))\|^2\\\\
   &\longrightarrow 0 \text{ as } n\longrightarrow\infty.
\end{array}
$$

Thus, $\Pi_2$ is continuous.

  \item $\Pi_2$ maps  $\mathcal{B}_k$ into equicontinuous  family.\\
  Let $z\in \mathcal{B}_k$ and $\tau_1,\,\tau_2\in[0,T]$, $\tau_1,\,\tau_2\neq
  t_k$, k=1,...,m,  we have
  $$
  \begin{array}{ll}
    \E\|(\Pi_2z)(\tau_2)-(\Pi_2z)(\tau_1)\|^2 \leq & 6\E\|\int_0^{\tau_1} (S(\tau_2-s)-S(\tau_1-s))f(s,z_s+y_s)ds\|^2
    \\\\
     & +6\E\|\int_0^{\tau_1} (S(\tau_2-s)-S(\tau_1-s))Bu(s)ds\|^2\\\\
     &+6\E\|\sum_{0<t_k<\tau_1} (S(\tau_2-t_k)-S(\tau_1-t_k))
     I_k(z({t_k^-})+y({t_k^-}))\|^2 \\\\
     & +6\E\|\int_{\tau_1}^{\tau_2} S(\tau_2-s)f(s,z_s+y_s)ds\|^2\\\\
     &+6\E\|\int_{\tau_1}^{\tau_2} S(\tau_2-s))Bu(s)ds\|^2
     \\\\
     & +6\E\|\sum_{\tau_1<t_k<\tau_2} S(\tau_2-t_k)
     I_k(z({t_k^-})+y({t_k^-}))\|^2.
  \end{array}
  $$

By the inequality \ref{estim0}  and H\"{o}lder inequality, we get
$$
 \begin{array}{ll}
    \E\|(\Pi_2z)(\tau_2)-(\Pi_2z)(\tau_1)\|^2 \leq & 6T\E\|\int_0^{\tau_1}
    \|S(\tau_2-s)-S(\tau_1-s)\|^2h_{q'}(s)ds
    \\\\
     & +6TM_b\mathcal{G}\int_0^{\tau_1} \|S(\tau_2-s)-S(\tau_1-s)\|^2ds\\\\
     &+6m \sum_{0<t_k<\tau_1} \|S(\tau_2-t_k)-S(\tau_1-t_k)\|^2\widetilde{M}_k  \\\\
     & +6T\int_{\tau_1}^{\tau_2} \|S(\tau_2-s)\|^2h_{q'}(s)ds+6TM_b\mathcal{G}\int_{\tau_1}^{\tau_2}
     \|S(\tau_2-s)\|^2ds
     \\\\
     & +6Mm\sum_{\tau_1<t_k<\tau_2}  \widetilde{M}_k.
  \end{array}
  $$
  The right-hand side is independent of $z\in \mathcal{B}_k$ and tends to zero
  as $\tau_2-\tau_1\longrightarrow0,$ since the compactness of
  $S(t)$ for $t>0$ implies the continuity in the uniform operator
  topology. Thus, $\Pi_2$ maps $\mathcal{B}_k$ into an equicontinuous family
  of   functions. The equicontinuities for the cases $\tau_1<\tau_2\leq0$
  and $\tau_1<0<\tau_2$ are obvious.\\

  \item $(\Pi \mathcal{B}_q)(t)$ is precompact set in $X$.\\
  Let $0<t\leq T$ be fixed,  $0<\epsilon<t$,  for $z\in \mathcal{B}_k$,  we define

  $$
  \begin{array}{ll}
    (\Pi_{2,\epsilon}z)(t)= & S(\epsilon)\int_0^{t-\epsilon}S(t-s-\epsilon)f(s,z_s+y_s)ds+S(\epsilon)\int_0^{t-\epsilon}S(t-s-\epsilon)Bu(s)ds\\\\
     & +S(\epsilon)\sum_{0<t_k<t-\epsilon}S(t-t_k-\epsilon)I_k(z({t_k^-})+y({t_k^-})).
  \end{array}
  $$
Using the estimation (\ref{estim0}) as above and by the compactness
of $S(t)$ $(t>0)$, we obtain
$V_\epsilon(t)=\{(\Pi_{2,\epsilon}z)(t):\; z\in \mathcal{B}_k\}$  is
relative compact in $X$ for every $\epsilon$, $0<\epsilon<t$.
Moreover, for every $z\in \mathcal{B}_k$, we have
$$
\begin{array}{ll}
    \E\|\Pi_{2}z)(t)-\Pi_{2,\epsilon}z)(t)\|^2\leq&  3T\int_{t-\epsilon}^t\|S(t-s)\|^2\E\|f(s,z_s+y_s)\|^2ds
    \\\\
   & +3TM_b\mathcal{G}\int_{t-\epsilon}^t\|S(t-s)\|^2ds \\\\
   & +3m
   \sum_{t-\epsilon<t_k<t}\|S(t-t_k)\|^2\E\|I_k(z({t_k^-})+y({t_k^-}))\|^2\\\\
   &\leq
   3TM\int_{t-\epsilon}^th_{q'}(s)ds+3TM_b\mathcal{G}M\epsilon\\\\
   &+3mM\sum_{t-\epsilon<t_k<t}\widetilde{M}_k.
\end{array}
$$

Therefore,
$$
\E\|\Pi_{2}z)(t)-\Pi_{2,\epsilon}z)(t)\|^2\longrightarrow0,\qquad
\text{ as } \epsilon\longrightarrow 0^+,
$$
and there are  precompact sets arbitrarily close to the set
$V(t)=\{(\Pi_2z)(t):\; z\in B_k\}$, hence the set $V(t)$ is also
precompact in $X$.

Thus, by Arzela-Ascoli theorem $\Pi_2$ is a compact operator. These
arguments enable us to conclude that $\Pi_2$  is completely
continuous, and by the fixed point theorem of Karasnoselskii there
exists a fixed point $z(.)$ for $\widehat{\Pi}$ on $\mathcal{B}_k$.
If we define $x(t)=z(t)+y(t),$ $-\infty<t\leq T$, it is easy to see
that $x(.)$ is a mild solution  of (\ref{eq1}) satisfying
$x_0=\varphi$, $x(T)=x_1$. Then the proof is complete.
\end{enumerate}

\end{proof}

\section{Example}

 To illustrate the previous result, we consider the following   impulsive neutral stochastic partial differential equation with infinite
 delays,
driven by a fractional Brownian motion of the form\\
   \small{
   \begin{equation}\label{eq1part}
 \left\{\begin{array}{llll}
\frac{\partial}{\partial
t}[v(t,\xi)-\int_{-\infty}^0T(v(\theta,\xi))G(t,\xi,\theta-t)d\theta]=[\frac{\partial^2}{\partial^2\xi}
v(t,\xi)+c(\xi)u(t)\\\\+\int_{-\infty}^0Q(v(\theta,\xi))F(t,\xi,\theta-t)d\theta]
 +\sigma (t)\frac{dB^H(t)}{dt},\quad 0\leq t\leq T,\, t\neq t_k,\, 0\leq \xi\leq\pi\\\\
\Delta
v(t_k,\xi)=\int_{-\infty}^{t_k}\alpha_k(t_k^--s)r(v(s,\xi))ds,\;\quad
k=1,2,...,m;\\\\
v(t,0)=v(t,\pi)=0,\quad \quad 0\leq t\leq T,\\\\
v(s,\xi)=\varphi(s,\xi) ,\,\;;-\infty< s \leq 0\quad 0\leq
\xi\leq\pi,
\end{array}\right.
\end{equation}
} where   $B^H(t)$ is  cylindrical  fractional Brownian motion,
$\varphi: [-\tau,0]\times[0,\pi]\longrightarrow\R$ is a given
continuous function such that $\varphi(s,.)\in L^2([0,\pi])$ is
measurable and satisfies $\E\|\varphi\|^2<\infty.$

We rewrite (\ref{eq1part}) into abstract form of (\ref{eq1}). We
take  $X=Y=U=L^2([0,\pi])$. Define the operator $A:D(A)\subset
X\longrightarrow X$ given  by $A=\frac{\partial^2}{\partial^2\xi}$
with
$$
D(A)=\{y\in X:\,y' \mbox{ is absolutely continuous},  y''\in X,\quad
y(0)=y(\pi)=0\},
$$
then we get
$$
Ax=\sum_{n=1}^\infty n^2<x,e_n>_Xe_n,\quad x\in D(A),
$$
 where $
e_n:=\sqrt{\frac{2}{\pi}}\sin nx,\; n=1,2,.... $
 is  an orthogonal  set of eigenvector of $-A$.\\

 The bounded linear operator $(-A)^{\frac{3}{4}}$ is given by
 $$
(-A)^{\frac{3}{4}}x=\sum_{n=1}^\infty
n^\frac{3}{2}<x,e_n>_Xe_n,\quad
 $$
 with domain
 $$
 D((-A)^{\frac{3}{4}})=X_{\frac{3}{4}}=\{x\in X,\sum_{n=1}^\infty
n^\frac{3}{2}<x,e_n>_Xe_n\in X \},\,\text{ and }
\|(-A)^{\frac{3}{4}}\|=1.
$$
 It is well known that $A$ is the infinitesimal generator of an
analytic semigroup $\{S(t)\}_{t\geq 0}$ in $X$, and is given by (see
\cite{pazy})
$$ S(t)x=\sum_{n=1}^{\infty}e^{-n^2t}<x,e_n>e_n,
$$
for $x\in X$ and $t\geq0$.  Since the semigroup $\{S(t)\}_{t\geq 0}$
is analytic, there exists a constant $M>0$  such that
$\|S(t)\|^2\leq M$ for every $t\geq0$. In other words, the condition
$(\mathcal{H}.1)$ holds.

\ni Further for  the operator $B: U\longrightarrow X$ is a bounded
linear operator defined by
  $
Bu(t)(\xi)=c(\xi)u(t),\;0\leq \xi \leq\pi, \, u\in L^2([0,T], U),
  $ and    the operator $W:L^2([0,T],
U)\longrightarrow X $ given by
  $$
W u(\xi)=\int_0^TS(T-s)c(\xi)u(t)ds,\;\;0\leq \xi \leq\pi,
  $$
  $W$ is a bounded linear operator but not necessarily one-to-one.
  Let
  $$Ker \,W=\{x\in L^2([0,T],U),  \; W x=0\}
  $$
   be the null space of
  $W$ and $[Ker\, W]^\bot$ be its orthogonal complement in
  $L^2([0,T],U)$. Let $\widetilde{W}: [Ker\, W]^\bot\longrightarrow Range(W)
  $ be the restriction of $W$ to $[Ker\, W]^\bot$, $\widetilde{W}$ is
  necessarily one-to-one operator. The inverse mapping theorem  says that
  $\widetilde{W}^{-1}$ is bounded since $[Ker\, W]^\bot$ and $Range
  (W)$ are Banach spaces. So that  $W^{-1}$ is bounded and takes values in  $L^2([0,T],U)\setminus Ker\, W$,  hypothesis  $(\mathcal{H}.6)$ is satisfied.

We choose the phase function $h(s)=e^{4s}$, $s<0$, then
$l=\int_{-\infty}^0h(s)ds=\frac{1}{4}<\infty$, and the abstract
phase space $\mathcal{B}_h$ is Banach space with the norm
 $$
\|\varphi\|_{\mathcal{B}_h}=\int_{-\infty}^0h(s)\sup_{\theta\in[s,0]}(\E\|\varphi(\theta)\|^2)^\frac{1}{2}ds.
$$

To rewrite the initial-boundary value problem (\ref{eq1part}) in the
abstract form we assume the following:

For $(t,\varphi)\in[0,T]\times\mathcal{B}_h$, where
$\varphi(\theta)(\xi)=\varphi(\theta,\xi),$
$(\theta,\xi)\in(-\infty,0]\times[0,\pi]$, we put
$v(t)(\xi)=v(t,\xi)$

\begin{align*}
 g(t,\varphi)(\xi)&=\int_{-\infty}^0T(\varphi(\theta)(\xi))G(t,\xi,\theta-t)d\theta,\\\\
  f(t,\varphi)(\xi)&=
  \int_{-\infty}^0Q(\varphi(\theta)(\xi))F(t,\xi,\theta-t)d\theta.
\end{align*}

The above equation can be written in the abstract form (\ref{eq1}).

We now assume that the functions $\alpha_k:\R\longrightarrow\R,$
$k=1,2,...,m$ are continuous and
$\widetilde{M}_k=\int_{-\infty}^{t_k}h(s)\alpha_k^2(s)ds<\infty$.
Then the condition $(\mathcal{H}.4)$  is satisfied if $r(.)$ is
Lipschitz continuous. To verify the conditions $(\mathcal{H}.2)$ and
$(\mathcal{H}.3)$, we suppose further that
\begin{enumerate}
  \item[(i)] the function $F(t,\xi,\theta)$ is continuous in
  $[0,T]\times[0,\pi]\times(-\infty,0]$ and satisfies
  $$
\int_{0}^\pi(\int_{-\infty}^0|F(t,\xi,\theta)|d\theta)^2<\infty.
  $$
  \item [(ii)] the function $Q(.)$ is continuous and $\E
  Q^2(\varphi(\theta)(\xi))\leq \|\varphi\|^2_{\mathcal{B}_h}$, for
  all $(\theta,\xi)\in(-\infty,0]\times[0,\pi]$.
  \item [(iii)] the function $Q(.)$ is continuous and $\E\|
  Q(\varphi_1(\theta)(\xi))-Q(\varphi_2(\theta)(\xi))\|^2\leq\|\varphi_1-\varphi_2\|_{\mathcal{B}_h}^2.$
\end{enumerate}

We can see from $(i)$ and $ (ii)$ that
$$
\E\|F(t,\varphi)\|^2\leq
\int_{0}^\pi(\int_{-\infty}^0|F(t,\xi,\theta)|d\theta)^2d\xi\|\varphi\|_{\mathcal{B}_h}^2,
$$
which implies that  $(\mathcal{H}.2)$ is satisfied. Moreover

$$
g([0,T]\times\mathcal{B}_h)\subseteq D((-A)^{\frac{3}{4}})
$$
and
$$
\E\|(-A)^{\frac{3}{4}}g(t,\varphi_1)-(-A)^{\frac{3}{4}}g(t,\varphi_2)\|^2\leq
M_g\|\varphi_1-\varphi_2\|_{\mathcal{B}_h}^2,
$$
for some constant $M_g$ depending on $T$ and $G$. Further, other
assumptions are satisfied such that
$$
\frac{7}{2}(1+8MM_bM_wT^2)\{(1+2M^2_{\frac{1}{4}}T^{\frac{3}{2}}))M_g+MT^2M_f\}<1.
$$

In order to define the operator $Q:
Y:=L^2([0,\pi],\R)\longrightarrow Y$, we choose a sequence
$\{\lambda_n\}_{n\in\N}\subset \R^+$, set $Qe_n=\lambda_ne_n$, and
assume that
$$
tr(Q)=\sum_{n=1}^\infty \sqrt{\lambda_n}<\infty.
$$
 Define the fractional Brownian motion in $Y$ by
$$
B^H(t)=\sum_{n=1}^\infty \sqrt{\lambda_n}\beta^H(t)e_n,
$$
where $H\in(\frac{1}{2},1)$ and $\{\beta^H_n\}_{n\in\N}$ is a
sequence of one-dimensional fractional Brownian motions mutually
independent. Let us assume the  function
$\sigma:[0,+\infty)\rightarrow
\mathcal{L}_2^0(L^2([0,\pi]),L^2([0,\pi]))$ satisfies
 $$\int_0^T\|\sigma(s)\|^2_{\mathcal{L}_2^0}ds< \infty,\;\; \forall T>0.
  $$

Then all the assumptions of Theorem \ref{th1} are satisfied.
Therefore, we conclude that the system (\ref{eq1part}) is
controllable on $(-\infty,T]$.





\begin{thebibliography}{30}
\bibitem {anguraj}{A. Anguraj, A. Vinodkumar. } Existence, uniqueness and stability
results of impulsive stochastic semilinear neutral functional
differential equations with infinite delays, Electron. J. Qual.
Theory Differ. Equ. Vol. 2009, No. 67 (2009),  1-13.



\bibitem  {boufoussi1}
{B. Boufoussi and  S. Hajji}, {Functional differential equations
driven by a fractional Brownian motion}, \textit{Computers and
Mathematics with Applications}, 62 (2011), 746-754.


\bibitem  {boufoussi2}
{B. Boufoussi,  S. Hajji, and  E. Lakhel  }, {Functional
differential equations in Hilbert spaces driven by a fractional
Brownian motion}, \textit{Afrika Matematika}, 23 (2) (2012),
173-194.

\bibitem  {boufoussi3}
{B. Boufoussi and  S.  Hajji}, {Neutral stochastic functional
differential equation driven by a  fractional  Brownian motion in a
Hilbert space},  \textit{Statist. Probab. Lett.}, 82 (2012),
1549-1558.

\bibitem  {bur-kir} T. A. Burton and C. A. Kirk, Fixed point theorem
of Karasnoselskii-Schaefer type. Math. Nachr., 189 (1998), 23-31.
\bibitem  {carab13}
{ T. Caraballo and  M.A.  Diop}, {Neutral stochastic delay partial
functional integro-differential equations driven by a fractional
Brownian motion}, \textit{Frontiers of Mathematics in China}, 8 (4)
 (2013), 745-760.

\bibitem  {carab}
{ T. Caraballo ,   M.J.  Garrido-Atienza, and   T. Taniguchi,} {The
existence and exponential behavior of solutions to stochastic delay
evolution equations with a fractional Brownian motion},
\textit{Nonlinear Analysis}, 74  (2011),  3671-3684.






   \bibitem  {dung15}
 {N. T. Dung},   { Stochstic Volterra integro-differential equations driven by
    by fractional Brownian motion in Hilbert space}, \textit{Stochastics},  87 (1)  (2015), 142-159.


\bibitem {gold}
    {G. Goldstein and A. Jerome},   \textit{Semigroups of linear operators and
    applications
     Oxford Mathematical Monographs}, The Clarendon Press, Oxford University Press, New
     York (1985).

     \bibitem{HalKat1}  J. K. Hale, and J. Kato, \emph{Phase space for retarded
equations with infinite delay}, Funkcial Ekvac., 21   (1978), 11-41.


\bibitem  {klam07}{J. Klamka,}  Stochastic controllability of linear systems
with delay in control, \textit{Bull. Pol. Acad. Sci. Tech. Sci.}, 55
 (2007), 23-29.

\bibitem  {klam13}{J. Klamka,} Controllability of dynamical systems. A survey.
\textit{Bull. Pol. Acad. Sci. Tech. Sci.},  61   (2013), 221-229.

\bibitem  {lak16} { E. Lakhel,} Controllability Of Neutral Stochastic Functional
Integro-Differential Equations Driven By Fractional Brownian Motion.
 Stochastic Analysis and Applications (To appear).

\bibitem  {lak15} { E. Lakhel and S. Hajji,  } Existence and Uniqueness   of Mild Solutions to  Neutral
SFDEs driven by a  Fractional  Brownian Motion  with Non-Lipschitz
Coefficients. Journal of Numerical Mathematics and Stochastics, 7
(1) (2015), 14-29.

\bibitem  {lak8} E. Lakhel, E.  \,and M. A.  McKibben,  Controllability of Impulsive
Neutral Stochastic  Functional Integro-Differential Equations Driven
by Fractional Brownian Motion. Chapter 8 In book : Brownian Motion:
Elements, Dynamics, and Applications.  Editors:  M. A. McKibben \&
M. Webster.  Nova Science Publishers, New York, 2015, pp. 131-148.

\bibitem  {li-li} {Y. Li and  B. Liu} Existence of solution of
nonlinear neutral functional differential inclusion with infinite
delay. Stoc. Anal. Appl. 25 (2007), 397-415.
\bibitem  {mahkaru14}
{ R. Maheswari  and S. Karunanithi} {Asymptotic stability of
stochastic impulsive neutral  partial functional differential
equations.} International J. of comp. Appli. No 18 (2014), 23-26.






\bibitem  {nualart}
{D. Nualart, }   \textit{ The Malliavin Calculus and Related Topics,
second edition,}   Springer-Verlag, Berlin (2006).

\bibitem  {pazy}
 { A. Pazy}, \textit{Semigroups of Linear Operators and Applications to Partial Differential Equations.
   Applied Mathematical Sciences, vol. 44},  Springer-Verlag, New York (1983).


\bibitem  {ren13}
{Y. Ren,  X. Cheng,  and R. Sakthivel,
  }{ On time-dependent stochastic evolution equations driven by
fractional Brownian motion in Hilbert space with finite delay,}
\textit{Mathematical methods in the Applied Sciences}, 37  (2013),
2177-2184.

\bibitem {ren11} {Y. Ren, L. Hu, and R. Sakthivel,} Controllability of impulsive neutral
stochastic functional differential inclusions with infinite delay,
\textit{J. Comput. Appl. Math.}, 235 (8)  (2011), 2603-2614.

\bibitem{ren13a}{R. Sakthivel, R. Ganesh, Y. Ren, and  S. M. Anthoni,} Approximate controllability of nonlinear fractional
dynamical systems, \textit{Commun. Nonlinear Sci. Numer. Simul.}, 18
  (2013), 3498-3508.
\bibitem{ren13b}{Y. Ren,
H. Dai, and  R. Sakthivel. } Approximate controllability of
stochastic differential system driven by a Levy process,
\textit{Internat. J. Control}, 86  (2013), 1158-1164.

\bibitem  {sakt12}
{R. Sakthivel, J.W. Luo.}  Asymptotic stability of impulsive
stochastic partial differential equations, Statist. Probab. Lett.,
79 (2009), 1219 - 1223.

\bibitem  {xu2006}
{D. Xu and  Z.Yang.} Exponential stability of nonlinear impulsive
neutral differential equations with delays, Nonlinear Anal. 67 (5)
(2006), 14261439.



\end{thebibliography}
\end{document}